\documentclass[preprint,superscriptaddress,nofootinbib,11pt]{revtex4-2}

\usepackage[T1]{fontenc}
\usepackage[utf8]{inputenc}
\usepackage{amsmath,amssymb,amsthm,mathtools,bm}
\usepackage{microtype}
\usepackage{enumitem}
\usepackage{graphicx}
\usepackage{tikz}
\usetikzlibrary{arrows.meta,positioning}
\usepackage{booktabs}
\usepackage{array}
\usepackage[hidelinks]{hyperref}

\newtheorem{theorem}{Theorem}[section]
\newtheorem{proposition}[theorem]{Proposition}
\newtheorem{corollary}[theorem]{Corollary}
\newtheorem{lemma}[theorem]{Lemma}
\theoremstyle{definition}
\newtheorem{definition}[theorem]{Definition}
\theoremstyle{remark}
\newtheorem{remark}[theorem]{Remark}

\newcommand{\R}{\mathbb{R}}
\newcommand{\one}{\mathbf{1}}
\newcommand{\Hzero}{\mathcal{H}_0}
\newcommand{\PiH}{\Pi_{\mathrm{Hodge}}}
\newcommand{\Picc}{\Pi_{\mathrm{cc}}}
\newcommand{\Mcc}{\mathcal{M}_{\mathrm{cc}}}
\newcommand{\ee}{\mathrm e}
\DeclareMathOperator{\imop}{im}

\providecommand{\authorcontributions}[1]{%
  \par\medskip\noindent\textbf{Author Contributions.} #1\par\medskip
}

\begin{document}

\title{\Large The Cactus Criterion: When Nonlinear Hodge Theory Reduces to Linear on Graphs}

\author{Sebastian Pardo-Guerra}
\affiliation{Recognition Physics Institute, Austin, Texas 78701, USA}
\affiliation{Center for Engineered Natural Intelligence, University of California, San Diego, CA 92093, USA}

\author{Anil Thapa}
\affiliation{Recognition Physics Institute, Austin, Texas 78701, USA}

\author{Jonathan Washburn}
\affiliation{Recognition Physics Institute, Austin, Texas 78701, USA}

\begin{abstract}
Let $G$ be a finite connected simple graph with a chosen orientation of its edges. For the edge potential
$\psi(t)=\cosh t-1,$
we minimize
$\sum_{e\in E^\to}\psi(z_e)$
over each affine class $\omega+dC^0(G)\subset C^1(G)$. The minimizer is the unique representative satisfying the nonlinear coclosed equation
$\delta\sinh z=0,$
and hence defines a nonlinear selector
$\Picc:C^1(G)\to C^1(G).$
We show that $\Picc$ is real analytic, identify its image as
$\imop \Picc=\Mcc=\operatorname{arsinh}(\ker\delta),$
and compute its differential as a weighted Hodge projector. In particular, $\Picc$ agrees with the ordinary Hodge projector $\PiH$ to first order at the origin, and the first nonlinear correction is cubic. Our main global theorem is a graph-theoretic criterion: for every admissible edge potential---even, $C^2$, strictly convex, and non-quadratic---the associated nonlinear selector coincides with $\PiH$ on all of $C^1(G)$ if and only if $G$ is a cactus graph. Finally, we work out the two-triangle graph, the smallest connected simple obstruction, and record a self-concordant Newton method for computing $\Picc$.
\end{abstract}

\maketitle 
\clearpage
\tableofcontents
\clearpage

\section{Introduction}

On a finite connected graph $G$ with a chosen orientation of its edges, the classical discrete Hodge theorem decomposes every edge field into the sum of a gradient and a coclosed remainder. Going back to Eckmann's combinatorial refinement of Hodge's decomposition theorem and to Dodziuk's finite-difference approach \cite{hodge1941,eckmann,dodziuk,dodziuk_fd}, this decomposition has become the linear workhorse of discrete calculus on graphs, underpinning the spectral theory of combinatorial Laplacians \cite{friedman,chung1997,horak_jost}, the graph-theoretic formulation of discrete exterior calculus \cite{grady,kozlov,lim}, and a growing list of applications in ranking, edge-flow analysis, and higher-order network models \cite{jiang_hodge,schaub_hodge}. Concretely, each affine class $[\omega]=\omega+dC^0(G)$ contains a unique representative $z$ satisfying the linear coclosed equation $\delta z=0$, and the assignment $\omega\mapsto z$ is the orthogonal Hodge projector $\PiH$. This $\PiH$ is the linear benchmark throughout the paper.

The present work asks what happens when the quadratic Dirichlet energy $\tfrac12\sum_e z_e^2$ underlying $\PiH$ is replaced by a genuinely nonlinear convex edge potential. Our principal example is
\begin{equation}
\psi(t)=\cosh t-1,
\label{eq:intro_psi}
\end{equation}
which under the change of variables $x=\ee^{t}$ is the reciprocal cost $J(x)=\tfrac12(x+x^{-1})-1$ recently characterized as canonical in~\cite{rs_cost}. The same cost serves as the keystone input of the cost-first discrete ledger framework of~\cite{PSTW2026Ledger}, where it governs balanced postings on a directed graph; here it reappears as the edge potential of a variational problem on the same graph topology. For our purposes the important features of $\psi$ are that it is even, strictly convex, real analytic, and explicit enough to admit closed-form manipulations. By strict convexity, the $\cosh$-energy $\sum_{e}(\cosh z_e-1)$ has a unique minimizer within each affine class $[\omega]$, and its Euler--Lagrange equation
\begin{equation}
\delta\sinh z=0
\label{eq:intro_coclosed}
\end{equation}
is the natural nonlinear analogue of $\delta z=0$. Writing $\Picc(\omega)$ for the unique element of $[\omega]$ that solves~\eqref{eq:intro_coclosed} defines a globally well-posed \emph{nonlinear selector} $\Picc\colon C^1(G)\to C^1(G)$, which plays the same structural role for $\psi$ that $\PiH$ plays in the quadratic theory.

The comparison of $\Picc$ and $\PiH$ is nontrivial for two opposing reasons. Near the origin the Taylor expansion $\cosh t-1=\tfrac12 t^2+\tfrac1{24}t^4+O(t^6)$ shows that $\Picc$ is a higher-order deformation of $\PiH$, so the two selectors are forced to agree at least to first order. Away from the origin, however, the hyperbolic sine in~\eqref{eq:intro_coclosed} bends the coclosed condition away from its linear counterpart, and no algebraic identity forces $\Picc=\PiH$ to persist. The paper is organized around the question that results: \emph{when does $\Picc$ coincide with $\PiH$, and how does the discrepancy first manifest itself when they differ?} We give a complete answer. The local part of the story (how $\Picc$ deviates from $\PiH$ to leading order) is worked out for the $\cosh$-potential above; the global part (on which graphs the two selectors agree everywhere) turns out to hold not only for $\psi(t)=\cosh t-1$ but for every \emph{admissible} edge potential, that is, every even $C^2$ strictly convex non-quadratic $\psi$.

Our main results fall into four parts. First, in Theorem~\ref{thm:companion}, every affine class contains exactly one nonlinear coclosed representative, obtained as the unique minimizer of the $\cosh$-energy on the mean-zero slice $\Hzero$; this minimizer is energetically separated from every other representative in the same class by an explicit quadratic gap inequality, which strengthens uniqueness to a quantitative stability statement. As a byproduct, the $\cosh$-energy is standard self-concordant, so a two-phase damped/full Newton method converges globally with an explicit iteration bound. Second, Theorems~\ref{thm:smoothness}--\ref{thm:weighted_hodge} show that the image of $\Picc$ is a smooth embedded submanifold of $C^1(G)$,
\begin{equation}
\imop\Picc=\Mcc=\operatorname{arsinh}(\ker\delta),
\label{eq:intro_image}
\end{equation}
of dimension $\beta_1(G)$, that $\Picc$ is real analytic, and that at every point $\omega$ its differential is a \emph{weighted} Hodge projector whose edge weights $w_e=\cosh(\Picc(\omega)_e)$ are determined by the output. The nonlinear selector therefore acts infinitesimally as a linear Hodge projector, but with weights that follow the solution itself. Third, in Theorem~\ref{thm:cubic} we compute the first nontrivial Taylor correction to the linear Hodge projector: the linear-order term of $\Picc$ at the origin is $\PiH$ itself, the second-order term vanishes identically because $\psi'=\sinh$ is odd, and the leading nonzero correction is cubic. Its coefficient is a closed-form expression built from $\PiH$, $L_0^{-1}$, and the coboundary $d$.

The fourth and headline result is the \emph{cactus criterion}. Recall that a connected graph is a \emph{cactus} if every edge lies on at most one simple cycle, equivalently if any two simple cycles share at most one vertex; cactus graphs are the Husimi trees originally introduced in Husimi's study of cluster integrals in statistical mechanics and enumerated by Harary and Uhlenbeck \cite{husimi1950,harary,diestel}. Theorem~\ref{thm:cactus} asserts that for every admissible potential $\psi$,
\[
\Pi_\psi=\PiH \text{ on all of } C^1(G)
\ \ \Longleftrightarrow\ \ 
G \text{ is a cactus graph.}
\]
The criterion is thus \emph{universal in $\psi$}: the same purely graph-theoretic condition controls the collapse for every admissible potential simultaneously. Intuitively, on a cactus the simple cycles occupy edge-disjoint supports, so odd nonlinearities act on each cycle independently and cannot couple them; on a non-cactus graph, any two cycles sharing an edge feed through the nonlinearity and produce an unavoidable obstruction. A worked computation on the two-triangle graph, the smallest connected simple non-cactus graph, makes this obstruction explicit: the nonlinear coclosed system reduces there to a single transcendental equation in one scalar unknown and can be compared numerically to its linear counterpart.

The remainder of the paper is organized as follows. Section~\ref{sec:setup} fixes the discrete-calculus notation, proves the variational existence theorem, and establishes the self-concordance bound that underlies the certified Newton method. Section~\ref{sec:projection} packages the minimizer assignment into the global map $\Picc$, identifies its image as the submanifold~\eqref{eq:intro_image}, and computes its differential. Section~\ref{sec:cubic} derives the cubic expansion of $\Picc$ near the origin. Section~\ref{sec:cactus} introduces admissible potentials and proves the cactus criterion. Finally, Section~\ref{sec:example} illustrates the whole picture on the two-triangle graph.

\section{Setup and the existence theorem}
\label{sec:setup}

This section has two goals. First, we fix the discrete-calculus conventions and recall the classical linear Hodge decomposition on a connected graph; this decomposition realizes the benchmark projector $\PiH$ against which the nonlinear theory is compared throughout. Second, we pose the variational problem for the $\cosh$-energy and prove that every cohomology class contains exactly one nonlinear coclosed representative. A short final subsection supplements the existence theorem with a self-concordance bound and a certified Newton method. Standard references for the discrete-calculus background are \cite{grady,lim}.

\subsection{Discrete calculus on a graph}
\label{sec:discrete_calculus}

Throughout the paper, $G=(V,E)$ is a finite, connected, simple, undirected graph, and $E^\to$ is a fixed orientation of its edges. The real vector spaces of vertex and edge functions are
\begin{equation}
C^0(G):=\R^V,
\qquad
C^1(G):=\R^{E^\to},
\end{equation}
each equipped with its standard Euclidean inner product; the associated norms are denoted $\|\cdot\|_{C^0}$ and $\|\cdot\|_{C^1}$. For $\xi,\zeta\in C^1(G)$ we write $(\xi\odot\zeta)(e):=\xi(e)\zeta(e)$ for edgewise multiplication, and $\xi^{\odot n}:=\xi\odot\cdots\odot\xi$ ($n$ factors). The vector $\one\in C^0(G)$ is the constant-one vertex function, and for $u\in C^0(G)$ we write $\bar u:=|V|^{-1}\sum_{v\in V}u_v$ for its mean.

The coboundary $d\colon C^0(G)\to C^1(G)$ and its $\ell^2$-adjoint $\delta\colon C^1(G)\to C^0(G)$ are defined by
\begin{equation}
(du)(e):=u(e_+)-u(e_-),
\qquad
(\delta\omega)(v):=\sum_{e_+=v}\omega(e)-\sum_{e_-=v}\omega(e),
\label{eq:coboundary}
\end{equation}
and satisfy the adjoint identity $\langle\omega,du\rangle_{C^1}=\langle\delta\omega,u\rangle_{C^0}$. The graph Laplacian is $L_G:=\delta d$. Most of the analysis below takes place on the mean-zero slice
\begin{equation}
\Hzero:=\Bigl\{u\in C^0(G):\sum_{v\in V}u_v=0\Bigr\},
\end{equation}
on which $L_G$ restricts to an isomorphism. A single structural observation about $\delta$ is used repeatedly throughout the paper.

\begin{lemma}
\label{lem:im_delta_hzero}
For every $\omega\in C^1(G)$ one has $\sum_{v\in V}(\delta\omega)(v)=0$, and hence $\imop\delta\subset\Hzero$.
\end{lemma}

\begin{proof}
Each oriented edge contributes $+1$ at its head and $-1$ at its tail, so the total sum telescopes to zero.
\end{proof}

\subsection{Linear Hodge decomposition}
\label{sec:linear_hodge}

Because $\delta$ maps into $\Hzero$ by Lemma~\ref{lem:im_delta_hzero}, the Laplacian restricts to an operator $L_0:=L_G|_{\Hzero}\colon\Hzero\to\Hzero$, and the linear Hodge projector
\begin{equation}
\PiH:=I-dL_0^{-1}\delta\colon C^1(G)\to C^1(G)
\end{equation}
is well-defined provided $L_0$ is invertible. The next lemma records the standard fact that it is, and identifies $\PiH$ as an orthogonal projector. Both the operators and the decomposition that follows will be used in essentially every proof below, so we include the short argument for completeness.

\begin{lemma}
\label{lem:linear_hodge}
The operator $L_0$ is an isomorphism of $\Hzero$, the Hodge decomposition
\begin{equation}
C^1(G)=\ker\delta\oplus dC^0(G)
\label{eq:hodge_decomp}
\end{equation}
holds, and $\PiH$ is the orthogonal projector onto $\ker\delta$ along $dC^0(G)$.
\end{lemma}

\begin{proof}
Connectedness implies $\ker(d|_{\Hzero})=\{0\}$, so for $u\in\Hzero\setminus\{0\}$,
\begin{equation}
\langle L_0u,u\rangle=\langle\delta du,u\rangle=\|du\|_{C^1}^2>0.
\end{equation}
Thus $L_0$ is positive definite and hence invertible on $\Hzero$. Setting $P:=I-dL_0^{-1}\delta$, any $\omega\in C^1(G)$ obeys
\begin{equation}
\delta(P\omega)=\delta\omega-\delta dL_0^{-1}\delta\omega
=\delta\omega-L_0L_0^{-1}\delta\omega=0,
\end{equation}
so $P\omega\in\ker\delta$. For $\omega=du$, writing $u_0:=u-\bar u\,\one\in\Hzero$ gives $P(du)=du-dL_0^{-1}L_0u_0=0$. Hence $P$ is the identity on $\ker\delta$ and annihilates $dC^0(G)$; it is therefore the orthogonal projector claimed.
\end{proof}

As an immediate consequence, the image of $\delta$ exhausts the mean-zero slice.

\begin{corollary}
\label{cor:im_delta}
$\imop\delta=\Hzero$, and in particular $\operatorname{rank}\delta=|V|-1$.
\end{corollary}

\begin{proof}
Lemma~\ref{lem:im_delta_hzero} gives the inclusion $\imop\delta\subset\Hzero$. Conversely, for any $y\in\Hzero$ the invertibility of $L_0$ supplies $u\in\Hzero$ with $\delta du=L_0u=y$, so $y\in\imop\delta$. The rank formula follows from $\dim\Hzero=|V|-1$.
\end{proof}

This completes the linear algebra. The first cohomology group $H^1(G;\R):=C^1(G)/dC^0(G)$ has dimension equal to the first Betti number
\begin{equation}
\beta_1(G):=|E|-|V|+1,
\end{equation}
and for $\omega\in C^1(G)$ we write $[\omega]:=\omega+dC^0(G)$ for its cohomology class. The linear Hodge projector $\PiH$ selects the unique coclosed representative of $[\omega]$; the rest of the section studies its nonlinear counterpart. In the discrete ledger of~\cite{PSTW2026Ledger}, the time-aggregated cycle-closure hypothesis (their Theorem~T3) restricts attention to flows lying entirely in the exact summand $dC^0(G)$, where the discrete Poincar\'e lemma reconstructs a scalar potential (their Theorem~T4); the Hodge decomposition $C^1(G)=\ker\delta\oplus dC^0(G)$ extends this picture to arbitrary cohomology classes by projecting onto the complementary coclosed summand, and $\PiH$ is precisely that projector. The remainder of this paper replaces the quadratic Dirichlet energy in each cohomology class by the cosh cost~$J$, and asks when the resulting nonlinear coclosed selector still coincides with $\PiH$.

\subsection{The nonlinear energy and the existence theorem}
\label{sec:nonlinear_energy}

Given an input edge field $h\in C^1(G)$, the $\cosh$-energy on the mean-zero slice is
\begin{equation}
\Phi_h(u):=\sum_{e\in E^\to}\bigl(\cosh(h(e)+(du)(e))-1\bigr),
\qquad u\in\Hzero.
\label{eq:sector_energy}
\end{equation}
Under the exponential change of variables $x_e(u):=\exp(h(e)+(du)(e))$ the edgewise summand equals the reciprocal cost $J(x_e)=\tfrac12(x_e+x_e^{-1})-1$ of~\cite{rs_cost}; $\Phi_h$ is thus the graph-level functional induced by $J$ on positive edge ratios. The exponential reformulation motivates the choice of potential but plays no role in the proofs, which we carry out entirely in the additive variables $z=h+du$.

The first variation of $\Phi_h$ is exactly the nonlinear coclosed equation announced in the introduction.

\begin{proposition}
\label{prop:stationarity}
For $h\in C^1(G)$ and $\phi\in\Hzero$, write $z=h+d\phi$. Then $D\Phi_h(\phi)[\eta]=0$ for all $\eta\in\Hzero$ if and only if $\delta\sinh z=0$.
\end{proposition}

\begin{proof}
For every $\eta\in\Hzero$,
\begin{equation}
D\Phi_h(\phi)[\eta]
=\sum_{e\in E^\to}\sinh\bigl(h(e)+(d\phi)(e)\bigr)\,(d\eta)(e)
=\langle\delta\sinh z,\eta\rangle_{C^0}.
\end{equation}
If $\delta\sinh z=0$ the pairing vanishes for every $\eta$. Conversely, if the pairing vanishes for every $\eta\in\Hzero$, then $\delta\sinh z$ is orthogonal to $\Hzero$; since $\delta\sinh z\in\Hzero$ by Lemma~\ref{lem:im_delta_hzero}, this forces $\delta\sinh z=0$.
\end{proof}

The equation $\delta\sinh z=0$ is the nonlinear analogue of the classical coclosed condition $\delta z=0$, and the selector $\Picc$ of Section~\ref{sec:projection} is defined by picking out the unique representative of $[h]$ that satisfies it. To prove that such a representative exists and is unique in each class, we need two scalar inputs: a Poincar\'e inequality on $\Hzero$, and a strong-convexity bound for the scalar potential $\cosh-1$.

\begin{lemma}
\label{lem:poincare}
There exists a constant $c_G>0$, depending only on $G$, such that $\|du\|_{C^1}\ge c_G\|u\|_{C^0}$ for every $u\in\Hzero$.
\end{lemma}

\begin{proof}
The map $d|_{\Hzero}$ is injective: $du=0$ forces $u$ to be constant on the connected graph $G$, and the mean-zero condition then gives $u=0$. Injectivity of a linear map on a finite-dimensional space yields the stated norm bound.
\end{proof}

\begin{lemma}
\label{lem:strong_convexity}
For all $a,b\in\R$,
\begin{equation}
(\cosh a-1)-(\cosh b-1)-\sinh(b)(a-b)\ge \tfrac12(a-b)^2.
\label{eq:strong_convexity}
\end{equation}
\end{lemma}

\begin{proof}
The function $g(t):=\cosh t-1-\tfrac12 t^2$ has $g''(t)=\cosh t-1\ge 0$, so $g$ is convex. Consequently $g(a)\ge g(b)+g'(b)(a-b)$, and expanding $g$ and $g'$ yields~\eqref{eq:strong_convexity}.
\end{proof}

These ingredients combine with Proposition~\ref{prop:stationarity} to give the central existence theorem of the section.

\begin{theorem}
\label{thm:companion}
For every $h\in C^1(G)$, the functional $\Phi_h$ has a unique minimizer $u_h\in \Hzero$. The associated edge field
\begin{equation}
\omega_h:=h+du_h
\end{equation}
is the unique representative of the cohomology class $[h]$ satisfying
\begin{equation}
\delta\sinh \omega_h=0.
\label{eq:coclosed_eq}
\end{equation}
Moreover, for every $u\in \Hzero$,
\begin{equation}
\Phi_h(u)-\Phi_h(u_h)
\ge \frac12\sum_{e\in E^\to}\bigl(d(u-u_h)\bigr)(e)^2.
\label{eq:quadratic_gap}
\end{equation}
\end{theorem}

\begin{proof}
\emph{Strict convexity.} For $u,\eta\in\Hzero$,
\begin{equation}
D^2\Phi_h(u)[\eta,\eta]
=\sum_{e\in E^\to}\cosh\bigl(h(e)+(du)(e)\bigr)\,(d\eta)(e)^2
\ge \|d\eta\|_{C^1}^2,
\end{equation}
which by Lemma~\ref{lem:poincare} vanishes only for $\eta=0$. Hence $\Phi_h$ is strictly convex on $\Hzero$.

\emph{Coercivity.} The pointwise bound $\cosh t-1\ge\tfrac12 t^2$ gives
\begin{equation}
\Phi_h(u)
\ge \tfrac12\|h+du\|_{C^1}^2
\ge \tfrac12\bigl(c_G\|u\|_{C^0}-\|h\|_{C^1}\bigr)^2,
\end{equation}
so $\Phi_h(u)\to\infty$ as $\|u\|_{C^0}\to\infty$ in $\Hzero$. Being continuous on a finite-dimensional space, $\Phi_h$ attains its infimum; strict convexity then forces the minimizer to be unique.

\emph{Euler--Lagrange characterization.} By Proposition~\ref{prop:stationarity}, the minimizer $u_h$ satisfies~\eqref{eq:coclosed_eq}. Conversely, any $z\in[h]$ with $\delta\sinh z=0$ can be written $z=h+du$ after normalizing $u$ to lie in $\Hzero$, at which point Proposition~\ref{prop:stationarity} says $u$ is a critical point of $\Phi_h$; strict convexity forces $u=u_h$.

\emph{Quadratic gap.} Applying Lemma~\ref{lem:strong_convexity} edgewise with $a_e=h(e)+(du)(e)$ and $b_e=\omega_h(e)$ and summing over $e\in E^\to$,
\begin{equation}
\Phi_h(u)-\Phi_h(u_h)
\ge
\sum_{e}\sinh(\omega_h(e))\,d(u-u_h)(e)
+\tfrac12\sum_{e}\bigl(d(u-u_h)\bigr)(e)^2.
\end{equation}
The first sum equals $\langle\sinh\omega_h,d(u-u_h)\rangle_{C^1}=\langle\delta\sinh\omega_h,u-u_h\rangle_{C^0}=0$ by the adjoint identity and~\eqref{eq:coclosed_eq}, leaving exactly~\eqref{eq:quadratic_gap}.
\end{proof}

Theorem~\ref{thm:companion} is the variational backbone of the paper: every cohomology class contains a unique nonlinear coclosed representative, and the $\cosh$-energy distinguishes it from every other representative by a clean quadratic lower bound that sharpens uniqueness into a quantitative stability statement.

\subsection{Self-concordance and a certified Newton method}
\label{sec:newton}

Theorem~\ref{thm:companion} tells us the minimizer $u_h$ exists and is unique but gives no recipe for computing it. We close the section by showing that the pointwise separable structure of $\Phi_h$---a sum of identical one-dimensional convex functions, each composed with a linear edgewise map---produces a sharp self-concordance bound, which in turn yields a damped/full Newton method with an explicit iteration certificate. The argument follows the standard framework of Nesterov and Nemirovski~\cite{nesterov_nemirovski,nesterov,boyd_vandenberghe2004}.

\begin{theorem}
\label{thm:self_concordance}
For every $h\in C^1(G)$ and every $\phi,\eta\in \Hzero$,
\begin{equation}
\bigl|D^3\Phi_h(\phi)[\eta,\eta,\eta]\bigr|
\le
\bigl(D^2\Phi_h(\phi)[\eta,\eta]\bigr)^{3/2}.
\label{eq:self_concordance}
\end{equation}
In particular, $\Phi_h$ is standard self-concordant.
\end{theorem}

\begin{proof}
Write $c_e:=h(e)+(d\phi)(e)$ and $v_e:=(d\eta)(e)$, so that
\begin{equation}
D^2\Phi_h(\phi)[\eta,\eta]=\sum_{e\in E^\to}\cosh(c_e)\,v_e^2,
\qquad
D^3\Phi_h(\phi)[\eta,\eta,\eta]=\sum_{e\in E^\to}\sinh(c_e)\,v_e^3.
\end{equation}
Set $a_e:=\cosh(c_e)\,v_e^2\ge 0$. Since $|\sinh(c_e)|\le\cosh(c_e)\le\cosh(c_e)^{3/2}$,
\begin{equation}
|\sinh(c_e)|\,|v_e|^3\le a_e^{3/2}.
\end{equation}
For any finite sequence of nonnegative numbers one has $a_e^{3/2}\le a_e(\sum_j a_j)^{1/2}$; summing over $e$ gives $\sum_e a_e^{3/2}\le(\sum_e a_e)^{3/2}$, and combining these bounds yields~\eqref{eq:self_concordance}. This is the standard self-concordance inequality with constant $1\le 2$, so $\Phi_h$ is standard self-concordant in the sense of~\cite{nesterov_nemirovski}.
\end{proof}

This bound makes the textbook Newton theory applicable to $\Phi_h$. Writing $g_k:=\nabla\Phi_h(\phi^{(k)})$ and $H_k:=D^2\Phi_h(\phi^{(k)})$ (both on $\Hzero$), the \emph{Newton decrement} at iterate $k$ is
\begin{equation}
\lambda_k:=\bigl(g_k^{\top}H_k^{-1}g_k\bigr)^{1/2}.
\end{equation}
The standard damped/full Newton method, adapted to self-concordant functions, alternates between a globalizing phase in which $\lambda_k$ is large and a local phase in which full Newton steps converge quadratically.

\begin{theorem}
\label{thm:newton}
Fix $h\in C^1(G)$ and $\phi^{(0)}\in\Hzero$, and iterate
\begin{equation}
\phi^{(k+1)}=
\begin{cases}
\phi^{(k)}-\dfrac{1}{1+\lambda_k}H_k^{-1}g_k, & \lambda_k\ge \dfrac14,\\[1.2ex]
\phi^{(k)}-H_k^{-1}g_k, & \lambda_k<\dfrac14.
\end{cases}
\label{eq:damped_newton}
\end{equation}
Then
\begin{enumerate}[label=\textup{(\alph*)}]
\item the number of damped steps before $\lambda_k<1/4$ is at most
\begin{equation}
\left\lceil\frac{\Phi_h(\phi^{(0)})-\inf_{\Hzero}\Phi_h}{\theta_\star}\right\rceil,
\qquad
\theta_\star:=\tfrac14-\log\tfrac54;
\end{equation}
\item once $\lambda_k<1/4$, the full Newton phase converges quadratically, so the number of additional iterations needed to reach $\lambda_k<\varepsilon$ is $O(\log\log(1/\varepsilon))$.
\end{enumerate}
\end{theorem}

\begin{proof}
Self-concordance (Theorem~\ref{thm:self_concordance}) together with strict convexity of $\Phi_h$ on $\Hzero$ let us invoke the classical Nesterov--Nemirovski machinery~\cite{nesterov,nesterov_nemirovski}. Each damped step decreases the objective by at least $\theta(\lambda_k):=\lambda_k-\log(1+\lambda_k)\ge\theta_\star$ whenever $\lambda_k\ge 1/4$, which gives (a) after summing over damped steps. Once $\lambda_k<1/4$, the full-step local quadratic-convergence theorem for self-concordant functions~\cite{nesterov} yields (b).
\end{proof}

Theorems~\ref{thm:companion}, \ref{thm:self_concordance}, and~\ref{thm:newton} together establish the variational, analytic, and computational core of the paper: every cohomology class has a unique nonlinear coclosed representative, this representative is separated energetically from the rest of its class, and it can be computed by a globally convergent Newton scheme with an explicit iteration bound. The next section turns this pointwise existence theorem into a globally defined nonlinear selector and studies its structure.

\section{The nonlinear coclosed projection}
\label{sec:projection}

Theorem~\ref{thm:companion} assigns to each input edge field $\omega\in C^1(G)$ a distinguished representative of its cohomology class, namely the unique nonlinear coclosed one. Packaged as a function of $\omega$, this assignment defines a map
\begin{equation}
\Picc\colon C^1(G)\longrightarrow C^1(G),
\label{eq:picc_global_map}
\end{equation}
the nonlinear analogue of the Hodge projector $\PiH$. This section studies its structure. After recording the elementary invariance properties, we identify the image of $\Picc$ as a smooth embedded submanifold of $C^1(G)$ diffeomorphic to $\ker\delta$ through the edgewise $\operatorname{arsinh}$ map, and then compute the differential of $\Picc$ as a \emph{weighted} Hodge projector whose edge weights are determined by the output. The section closes by upgrading smoothness to real analyticity and specializing the differential formula at the origin, where it reduces exactly to $\PiH$.

\subsection{Definition, basic properties, and the image}
\label{sec:picc_def_image}

\begin{definition}
\label{def:picc}
For $\omega\in C^1(G)$, let
\begin{equation}
\phi^*(\omega):=\arg\min_{\phi\in\Hzero}\Phi_\omega(\phi),
\qquad
\Picc(\omega):=\omega+d\phi^*(\omega).
\label{eq:picc_def}
\end{equation}
\end{definition}

Three elementary properties of $\Picc$ follow directly from Theorem~\ref{thm:companion}: it is the nonlinear coclosed selector of its cohomology class, it respects exact shifts, and it is idempotent.

\begin{proposition}
\label{prop:basic_props}
For every $\omega\in C^1(G)$, the field $\Picc(\omega)$ is the unique element of $[\omega]$ satisfying $\delta\sinh\Picc(\omega)=0$. Moreover, $\Picc(\omega+d\chi)=\Picc(\omega)$ for every $\chi\in C^0(G)$, and $\Picc\circ\Picc=\Picc$.
\end{proposition}

\begin{proof}
The existence and uniqueness statement is Theorem~\ref{thm:companion} with $h=\omega$. Exact-shift invariance is immediate from the fact that $\omega$ and $\omega+d\chi$ define the same cohomology class. Idempotence follows because $\Picc(\omega)$ already satisfies the nonlinear coclosed condition and therefore is its own coclosed representative.
\end{proof}

Proposition~\ref{prop:basic_props} shows that $\Picc$ takes values in the \emph{nonlinear coclosed slice}
\begin{equation}
\Mcc:=\{z\in C^1(G):\delta\sinh z=0\},
\label{eq:Mcc_def}
\end{equation}
the global object of study for the rest of this section. The next result models $\Mcc$ as the image of $\ker\delta$ under the edgewise $\operatorname{arsinh}$.

\begin{proposition}
\label{prop:global_slice}
The edgewise $\operatorname{arsinh}$ map restricts to a smooth diffeomorphism
\begin{equation}
\operatorname{arsinh}\colon\ker\delta\xrightarrow{\sim}\Mcc,
\label{eq:arsinh_diffeo}
\end{equation}
and hence $\Mcc$ is a smooth embedded submanifold of $C^1(G)$ of dimension $\beta_1(G)$.
\end{proposition}

\begin{proof}
For $y\in\ker\delta$, the field $z:=\operatorname{arsinh}(y)$ satisfies $\delta\sinh z=\delta y=0$, so $z\in\Mcc$. Conversely, any $z\in\Mcc$ has $\sinh z\in\ker\delta$ and $z=\operatorname{arsinh}(\sinh z)$. Since $\sinh\colon\R\to\R$ is a smooth diffeomorphism, its edgewise inverse is a smooth diffeomorphism on $C^1(G)$; restricting to $\ker\delta$ gives~\eqref{eq:arsinh_diffeo}. The dimension assertion follows from $\dim\ker\delta=\beta_1(G)$.
\end{proof}

Combining the two propositions identifies the image of $\Picc$ with this submanifold.

\begin{corollary}
\label{cor:image_picc}
$\imop\Picc=\Mcc=\operatorname{arsinh}(\ker\delta)$.
\end{corollary}

\begin{proof}
The inclusion $\imop\Picc\subset\Mcc$ is Proposition~\ref{prop:basic_props}. Conversely, any $z\in\Mcc$ already satisfies $\delta\sinh z=0$, hence $\Picc(z)=z$ by idempotence, and so $z\in\imop\Picc$. The identification $\Mcc=\operatorname{arsinh}(\ker\delta)$ is Proposition~\ref{prop:global_slice}.
\end{proof}

Corollary~\ref{cor:image_picc} is the structural analogue of the linear identity $\imop\PiH=\ker\delta$: the nonlinear selector lands on $\operatorname{arsinh}(\ker\delta)$, a smooth submanifold of dimension $\beta_1(G)$ passing through the origin tangent to $\ker\delta$. The next subsection makes this tangency infinitesimally precise.

\subsection{Weighted Hodge decomposition}
\label{sec:weighted_hodge}

The differential of $\Picc$ at a point $\omega$ turns out to be a weighted analogue of the Hodge projector $\PiH$, with edge weights $w_e=\cosh(\Picc(\omega)_e)$ dictated by the output. We pause to develop this weighted calculus for arbitrary positive weights; the specialization $w=\cosh\Picc(\omega)$ will then yield the differential formula in Section~\ref{sec:smoothness}.

Fix $w\in(0,\infty)^{E^\to}$ and define the weighted divergence and weighted Laplacian
\begin{equation}
\delta_w\xi:=\delta(w\odot\xi),
\qquad
L_w:=(\delta_w d)|_{\Hzero}\colon\Hzero\to\Hzero,
\label{eq:weighted_ops}
\end{equation}
together with the weighted edge inner product and norm
\begin{equation}
\langle\xi,\zeta\rangle_w:=\sum_{e\in E^\to}w_e\,\xi(e)\,\zeta(e),
\qquad
\|\xi\|_w:=\langle\xi,\xi\rangle_w^{1/2}.
\label{eq:weighted_inner}
\end{equation}
(We use the same notation $\|\cdot\|_w$ for the induced operator norm on $\operatorname{End}(C^1(G))$.) A direct extension of Lemma~\ref{lem:linear_hodge} applies.

\begin{lemma}
\label{lem:weighted_laplacian}
The operator $L_w$ is self-adjoint and positive definite on $\Hzero$, and hence an isomorphism.
\end{lemma}

\begin{proof}
For $u,v\in\Hzero$,
\begin{equation}
\langle L_wu,v\rangle=\langle w\odot du,dv\rangle=\langle u,L_wv\rangle,
\end{equation}
and the diagonal gives $\langle L_wu,u\rangle=\sum_{e\in E^\to}w_e(du)_e^2\ge 0$. Vanishing of the right-hand side forces $du=0$, which by connectedness and the mean-zero constraint yields $u=0$. Hence $L_w$ is positive definite on the finite-dimensional space $\Hzero$, and therefore invertible.
\end{proof}

\begin{lemma}
\label{lem:weighted_hodge_decomposition}
For every $w\in(0,\infty)^{E^\to}$,
\begin{equation}
C^1(G)=\ker\delta_w\oplus dC^0(G),
\label{eq:weighted_hodge_decomp}
\end{equation}
with the summands orthogonal relative to $\langle\cdot,\cdot\rangle_w$. The operator $P_w:=I-dL_w^{-1}\delta_w$ is the weighted orthogonal projector onto $\ker\delta_w$ along $dC^0(G)$.
\end{lemma}

\begin{proof}
For $\xi\in\ker\delta_w$ and $\chi\in C^0(G)$ the adjoint identity gives $\langle\xi,d\chi\rangle_w=\langle\delta_w\xi,\chi\rangle=0$, so $\ker\delta_w$ and $dC^0(G)$ are weighted-orthogonal. Given $\xi\in C^1(G)$, Lemma~\ref{lem:im_delta_hzero} places $\delta_w\xi$ in $\Hzero$, and Lemma~\ref{lem:weighted_laplacian} produces $\eta:=L_w^{-1}\delta_w\xi\in\Hzero$; writing $\xi_{\mathrm{h}}:=\xi-d\eta$ one finds $\delta_w\xi_{\mathrm{h}}=\delta_w\xi-L_w\eta=0$. Hence $\xi=\xi_{\mathrm{h}}+d\eta$ with $\xi_{\mathrm{h}}\in\ker\delta_w$, proving~\eqref{eq:weighted_hodge_decomp}; directness follows from the weighted orthogonality.

Finally, $P_w$ is the identity on $\ker\delta_w$ because $\delta_w\xi=0$ kills the correction term, while on $d\chi$ the centered function $\chi_0:=\chi-\bar\chi\,\one\in\Hzero$ gives $P_w(d\chi)=d\chi_0-dL_w^{-1}L_w\chi_0=0$. Thus $P_w$ is exactly the weighted orthogonal projector claimed.
\end{proof}

\subsection{Smoothness, differential, and analyticity}
\label{sec:smoothness}

With the weighted calculus of Section~\ref{sec:weighted_hodge} in hand, we can compute the derivative of $\Picc$ by applying the implicit function theorem to the nonlinear coclosed equation
\begin{equation}
\mathcal G(\omega,\phi):=\delta\sinh(\omega+d\phi)=0,
\qquad (\omega,\phi)\in C^1(G)\times\Hzero.
\label{eq:G_def}
\end{equation}
By Lemma~\ref{lem:im_delta_hzero}, $\mathcal G$ takes values in $\Hzero$; by Theorem~\ref{thm:companion}, for each $\omega$ there is a unique $\phi=\phi^*(\omega)\in\Hzero$ with $\mathcal G(\omega,\phi)=0$. The partial derivative of $\mathcal G$ in the $\phi$ slot is exactly a weighted Laplacian, with weight determined by the solution; this is what produces the differential formula.

\begin{theorem}
\label{thm:smoothness}
The map $\Picc\colon C^1(G)\to C^1(G)$ is smooth. For $\omega\in C^1(G)$, set $z:=\Picc(\omega)$ and $w:=\cosh z$. Then for every $\xi\in C^1(G)$,
\begin{equation}
D\Picc(\omega)[\xi]=\xi-dL_w^{-1}\delta_w\xi,
\label{eq:differential_formula}
\end{equation}
where $\delta_w$, $L_w$ use the weight $w=\cosh z$.
\end{theorem}

\begin{proof}
Fix $\omega$ and set $\phi^*=\phi^*(\omega)$. The partial derivative of $\mathcal G$ at $(\omega,\phi^*)$ in the $\phi$ slot is
\begin{equation}
D_\phi\mathcal G(\omega,\phi^*)[\eta]=\delta(w\odot d\eta)=L_w\eta,
\end{equation}
which by Lemma~\ref{lem:weighted_laplacian} is an isomorphism of $\Hzero$. The implicit function theorem~\cite{krantzparks} therefore gives a smooth dependence $\omega\mapsto\phi^*(\omega)$, and consequently $\Picc(\omega)=\omega+d\phi^*(\omega)$ is smooth.

Differentiating $\mathcal G(\omega,\phi^*(\omega))=0$ in the direction $\xi$ and writing $\eta=D\phi^*(\omega)[\xi]$, one obtains
\begin{equation}
\delta\bigl(w\odot(\xi+d\eta)\bigr)=0,
\qquad\text{i.e.,}\qquad
L_w\eta=-\delta_w\xi,
\end{equation}
so $\eta=-L_w^{-1}\delta_w\xi$ and $D\Picc(\omega)[\xi]=\xi+d\eta=\xi-dL_w^{-1}\delta_w\xi$.
\end{proof}

The right-hand side of~\eqref{eq:differential_formula} is exactly the weighted Hodge projector $P_w$ of Lemma~\ref{lem:weighted_hodge_decomposition}. Combined with the tangent-space description of $\Mcc$ (obtained by differentiating $\delta\sinh z=0$), this gives the following geometric reading: at each point, $\Picc$ acts infinitesimally as a weighted Hodge projector onto the tangent space of $\Mcc$.

\begin{theorem}
\label{thm:weighted_hodge}
With $z=\Picc(\omega)$ and $w=\cosh z$, the tangent space to $\Mcc$ at $z$ is
\begin{equation}
T_z\Mcc=\ker\delta_w,
\label{eq:tangent_space}
\end{equation}
and $D\Picc(\omega)$ is the orthogonal projector onto $T_z\Mcc$ along $dC^0(G)$ with respect to $\langle\cdot,\cdot\rangle_w$. In particular $\|D\Picc(\omega)\|_w\le 1$, with equality whenever $\beta_1(G)>0$.
\end{theorem}

\begin{proof}
Differentiating $\delta\sinh z=0$ at $z$ gives $D(\delta\sinh)(z)[\xi]=\delta(\cosh z\odot\xi)=\delta_w\xi$, whence~\eqref{eq:tangent_space}. Theorem~\ref{thm:smoothness} then identifies $D\Picc(\omega)$ with the weighted projector $P_w$, which by Lemma~\ref{lem:weighted_hodge_decomposition} is exactly the $\langle\cdot,\cdot\rangle_w$-orthogonal projector onto $\ker\delta_w=T_z\Mcc$. Orthogonal projectors have $\|\cdot\|_w$-norm at most $1$, with equality iff their range has positive dimension, and here $\dim\ker\delta_w=\beta_1(G)$.
\end{proof}

At the origin, the weights trivialize and the weighted Hodge projector reduces to the ordinary one.

\begin{corollary}
\label{cor:origin_linearization}
$D\Picc(0)=\PiH$; equivalently, $\Picc$ agrees with the ordinary Hodge projector to first order at the origin.
\end{corollary}

\begin{proof}
At $\omega=0$ one has $\Picc(0)=0$ and $w=\cosh 0\equiv 1$, so $\delta_w=\delta$ and $L_w=L_0$. Theorem~\ref{thm:smoothness} then gives $D\Picc(0)=I-dL_0^{-1}\delta=\PiH$.
\end{proof}

The smoothness of Theorem~\ref{thm:smoothness} actually upgrades to real analyticity, because the defining map $\mathcal G$ is itself real analytic in $(\omega,\phi)$.

\begin{proposition}
\label{prop:analytic_picc}
The map $\Picc\colon C^1(G)\to C^1(G)$ is real analytic.
\end{proposition}

\begin{proof}
The map $\mathcal G$ in~\eqref{eq:G_def} is real analytic, since $\sinh$ is entire and $d,\delta$ are linear. At every $\omega_0\in C^1(G)$, the partial derivative $D_\phi\mathcal G(\omega_0,\phi^*(\omega_0))=L_{w_0}$ with $w_0=\cosh(\omega_0+d\phi^*(\omega_0))$ is an isomorphism of $\Hzero$ by Lemma~\ref{lem:weighted_laplacian}. The analytic implicit function theorem~\cite{krantzparks} therefore supplies a real-analytic branch $\omega\mapsto\phi^*(\omega)$ near $\omega_0$; since $\omega_0$ was arbitrary, $\Picc(\omega)=\omega+d\phi^*(\omega)$ is real analytic on all of $C^1(G)$.
\end{proof}

Corollary~\ref{cor:origin_linearization} and Proposition~\ref{prop:analytic_picc} together prepare the ground for the local expansion: $\Picc$ is a real-analytic deformation of $\PiH$ that agrees with it to first order at the origin. The next section computes the leading correction.

\section{Local nonlinear correction}
\label{sec:cubic}

Corollary~\ref{cor:origin_linearization} and Proposition~\ref{prop:analytic_picc} together describe $\Picc$ at and near the origin to first order: $D\Picc(0)=\PiH$, and $\omega\mapsto\Picc(\omega)$ is real analytic. We now compute the leading correction to this linear picture. Because the underlying potential $\psi(t)=\cosh t-1$ is even, the selector $\Picc$ inherits an \emph{odd} symmetry, which forces every even-order Taylor coefficient at the origin to vanish. The first potentially nonzero correction is therefore cubic, and the rest of the section identifies it: first along lines and then, by polarization, as a symmetric trilinear form on $C^1(G)$.

\begin{proposition}
\label{prop:odd}
For every $\omega\in C^1(G)$, $\Picc(-\omega)=-\Picc(\omega)$.
\end{proposition}

\begin{proof}
If $z=\Picc(\omega)$, then $z\in[\omega]$ and $\delta\sinh z=0$. Since $\sinh$ is odd, the field $-z$ lies in $[-\omega]$ and satisfies $\delta\sinh(-z)=-\delta\sinh z=0$. By uniqueness (Theorem~\ref{thm:companion}) this forces $\Picc(-\omega)=-z$.
\end{proof}

Combining real analyticity, oddness, and $D\Picc(0)=\PiH$, the Taylor series of $\Picc$ at the origin contains only odd-order terms, and its linear term is $\PiH$. The next theorem derives the cubic one by specializing to the one-parameter family $\varepsilon\mapsto\Picc(\varepsilon\xi)$.

\begin{theorem}
\label{thm:cubic}
For every $\xi\in C^1(G)$, as $\varepsilon\to 0$,
\begin{equation}
\Picc(\varepsilon\xi)
=\varepsilon\,\PiH(\xi)
-\frac{\varepsilon^3}{6}\,dL_0^{-1}\delta\bigl((\PiH\xi)^{\odot 3}\bigr)
+O(\varepsilon^5).
\label{eq:cubic_expansion}
\end{equation}
\end{theorem}

\begin{proof}
Write $z(\varepsilon):=\Picc(\varepsilon\xi)$. By Proposition~\ref{prop:analytic_picc} this is real analytic near $\varepsilon=0$, and by Proposition~\ref{prop:odd} it satisfies $z(-\varepsilon)=-z(\varepsilon)$, so its Taylor expansion is purely odd:
\begin{equation}
z(\varepsilon)=\varepsilon z_1+\varepsilon^3 z_3+O(\varepsilon^5),
\label{eq:z_expansion}
\end{equation}
with $z_1=\PiH(\xi)$ by Corollary~\ref{cor:origin_linearization}. The difference $z(\varepsilon)-\varepsilon\xi$ lies in the linear subspace $dC^0(G)$ for every $\varepsilon$, so each of its Taylor coefficients lies in $dC^0(G)$ as well; in particular $z_3=d\phi_3$ for some $\phi_3\in\Hzero$.

Inserting the expansion $\sinh t=t+\tfrac16 t^3+O(t^5)$ into the nonlinear coclosed equation $\delta\sinh z(\varepsilon)=0$ gives
\begin{equation}
\sinh z(\varepsilon)
=\varepsilon z_1+\varepsilon^3\Bigl(z_3+\tfrac16 z_1^{\odot 3}\Bigr)+O(\varepsilon^5),
\end{equation}
and matching the $\varepsilon^3$ coefficient of $\delta\sinh z(\varepsilon)=0$ yields $\delta z_3+\tfrac16\delta(z_1^{\odot 3})=0$. Since $z_3=d\phi_3$ and $L_0=\delta d$ is invertible on $\Hzero$ (Lemma~\ref{lem:linear_hodge}), this rewrites as $L_0\phi_3=-\tfrac16\delta(z_1^{\odot 3})$, whence
\begin{equation}
\phi_3=-\tfrac16 L_0^{-1}\delta(z_1^{\odot 3}),
\qquad
z_3=-\tfrac16 dL_0^{-1}\delta(z_1^{\odot 3}).
\end{equation}
Substituting $z_1=\PiH(\xi)$ into~\eqref{eq:z_expansion} gives~\eqref{eq:cubic_expansion}.
\end{proof}

Polarization lifts the one-variable expansion of Theorem~\ref{thm:cubic} to a full multivariate Taylor formula at the origin, in which the cubic term appears as a symmetric trilinear form.

\begin{corollary}
\label{cor:multivar_cubic}
$D^2\Picc(0)=0$, and the third Fr\'echet derivative of $\Picc$ at the origin is the symmetric trilinear map
\begin{equation}
D^3\Picc(0)[\xi,\eta,\zeta]
=-dL_0^{-1}\delta\bigl((\PiH\xi)\odot(\PiH\eta)\odot(\PiH\zeta)\bigr).
\label{eq:third_derivative_formula}
\end{equation}
Consequently, as $\omega\to 0$ in $C^1(G)$,
\begin{equation}
\Picc(\omega)
=\PiH(\omega)
-\frac16\,dL_0^{-1}\delta\bigl((\PiH\omega)^{\odot 3}\bigr)
+O(\|\omega\|_{C^1}^5).
\label{eq:multivar_cubic}
\end{equation}
\end{corollary}

\begin{proof}
Real analyticity and oddness together kill every even-order derivative of $\Picc$ at $0$; in particular $D^2\Picc(0)=0$. The Taylor expansion at $0$ therefore reduces to
\begin{equation}
\Picc(\omega)
=\PiH(\omega)+\tfrac16 D^3\Picc(0)[\omega,\omega,\omega]+O(\|\omega\|_{C^1}^5),
\end{equation}
using Corollary~\ref{cor:origin_linearization} for the linear term. Evaluating along $\omega=\varepsilon\xi$ and matching the $\varepsilon^3$-coefficient with Theorem~\ref{thm:cubic} gives the diagonal
\begin{equation}
D^3\Picc(0)[\xi,\xi,\xi]=-dL_0^{-1}\delta\bigl((\PiH\xi)^{\odot 3}\bigr).
\end{equation}
Since a symmetric cubic form is determined by its diagonal, polarization delivers the symmetric trilinear extension~\eqref{eq:third_derivative_formula}; inserting this into the Taylor expansion yields~\eqref{eq:multivar_cubic}.
\end{proof}

The cubic correction~\eqref{eq:multivar_cubic} is the best local description of the difference $\Picc-\PiH$ available: it quantifies how much the nonlinear selector bends away from the linear Hodge projector at small scales, and it identifies the exact form of the leading deviation. Being infinitesimal, however, it cannot by itself decide when the two selectors agree globally; Taylor data at the origin says nothing about what happens far from it. The next section answers the global question by a sharp combinatorial criterion on~$G$.

\section{The cactus criterion}
\label{sec:cactus}

This section answers the global question left open by the local expansion of Section~\ref{sec:cubic}: for which graphs does the nonlinear selector $\Picc$ coincide with the linear Hodge projector $\PiH$ on \emph{all} of $C^1(G)$? The answer is purely combinatorial and, remarkably, independent of the choice of potential: for every admissible edge potential $\psi$, the equality $\Pi_\psi=\PiH$ holds globally if and only if $G$ is a \emph{cactus graph}, a graph whose simple cycles do not overlap along edges.

The argument has two sides. The forward direction, that every admissible $\Pi_\psi$ collapses to $\PiH$ on a cactus, rests on a structural feature of cacti: their simple cycles supply an edge-disjoint basis of $\ker\delta$, and an odd nonlinearity acts on each basis element independently without coupling cycles. The reverse direction, that non-cactus graphs always force a genuine departure, comes from an obstruction concentrated at a single vertex where two overlapping cycles split. We first record the combinatorial definitions and illustrate them on trees and single cycles, then build the cycle-form basis, introduce admissible potentials, prove the shared-path obstruction, and finally combine everything into the main theorem.

\subsection{Cactus graphs and two illustrative cases}
\label{sec:cactus_defs}

\begin{definition}
\label{def:cactus}
A connected graph is a \emph{cactus graph} if every edge belongs to at most one simple cycle, equivalently if any two simple cycles share at most one vertex.
\end{definition}

The equivalence of the two formulations is standard; we use the edge-based form throughout, because the forward direction of the criterion reduces to edge-disjointness of cycle supports. Two special families of cacti---trees and single cycles---show in transparent form the mechanism by which the cactus property forces $\Picc=\PiH$.

\begin{proposition}
\label{prop:tree}
If $G$ is a connected tree, then $\Picc(\omega)=0$ for every $\omega\in C^1(G)$.
\end{proposition}

\begin{proof}
A tree has $H^1(G;\R)=0$, so every $\omega\in C^1(G)$ lies in the single cohomology class represented by the zero field. Since $0$ is trivially nonlinear coclosed, uniqueness in Theorem~\ref{thm:companion} gives $\Picc(\omega)=0$.
\end{proof}

\begin{proposition}
\label{prop:cycle}
Let $G=C_n$ be the cycle graph with its edges oriented consistently around the cycle. For every $\omega\in C^1(C_n)$,
\begin{equation}
\Picc(\omega)(e_i)=\frac{1}{n}\sum_{j=1}^n\omega(e_j),
\qquad i=1,\ldots,n.
\label{eq:cycle_formula}
\end{equation}
\end{proposition}

\begin{proof}
Write $z=\Picc(\omega)=\omega+du$ with $u\in C^0(C_n)$. At each vertex of the oriented cycle, the nonlinear coclosed equation reads $\sinh z(e_i)=\sinh z(e_{i-1})$, and injectivity of $\sinh$ forces every $z(e_i)$ to equal a common number $a$. Summing $z(e_j)=\omega(e_j)+(du)(e_j)$ around the oriented cycle kills the exact part by telescoping, so $na=\sum_i\omega(e_i)$, whence $a=\tfrac{1}{n}\sum_j\omega(e_j)$.
\end{proof}

In both cases $\Picc$ coincides with $\PiH$, consistent with trees and single cycles being cactus graphs. The rest of the section extends this from the two degenerate families to every cactus, and, crucially, to every admissible potential $\psi$.

\subsection{Cycle-form basis on a cactus}
\label{sec:cycle_basis}

The forward direction of the cactus criterion rests on a single combinatorial fact: on a cactus, the simple cycles can be encoded as $\pm 1$-valued edge fields with pairwise disjoint supports, and these fields together form a basis of $\ker\delta$.

\begin{definition}
\label{def:cycle_form}
Let $C$ be a simple cycle in $G$ with a chosen orientation. The \emph{oriented cycle form} $h_C\in C^1(G)$ is defined by
\begin{equation}
h_C(e)=
\begin{cases}
+1 & \text{if $e$ lies on $C$ and agrees with the chosen orientation,}\\
-1 & \text{if $e$ lies on $C$ and disagrees with the chosen orientation,}\\
0  & \text{if $e$ does not lie on $C$.}
\end{cases}
\label{eq:cycle_form}
\end{equation}
\end{definition}

\begin{lemma}
\label{lem:cactus_basis}
Let $G$ be a connected cactus with simple cycles $C_1,\dots,C_m$, and let $h_j:=h_{C_j}$. Then
\begin{enumerate}[label=\textup{(\alph*)}]
\item the supports of the $h_j$ are pairwise edge-disjoint;
\item each $h_j$ lies in $\ker\delta$;
\item $m=\beta_1(G)$, so $\{h_1,\dots,h_m\}$ is a basis of $\ker\delta$.
\end{enumerate}
\end{lemma}

\begin{proof}
Part~(a) is the defining property of a cactus: any two simple cycles share at most one vertex and hence no edges. For~(b), each vertex of $C_j$ is incident to one incoming and one outgoing oriented edge of $C_j$ whose $\pm 1$ contributions cancel, so $\delta h_j=0$.

For~(c), the edge-disjoint supports make $\{h_1,\dots,h_m\}$ linearly independent, so it suffices to show $m=\beta_1(G)$. We induct on $m$: if $m=0$ the graph is a tree with $\beta_1(G)=0$; if $m\ge 1$, deleting any edge of $C_m$ keeps $G$ connected, preserves the cactus property, reduces the number of simple cycles to $m-1$, and drops $\beta_1$ by exactly one (one edge removed, no vertices). The inductive hypothesis then gives $\beta_1(G)=m$. Combined with the identity $\dim\ker\delta=\beta_1(G)$ for connected $G$, this proves that the linearly independent family of the right size is a basis.
\end{proof}

The edge-disjointness in (a) is the key feature that will make the forward direction of the cactus criterion work for any odd nonlinearity: different cycles cannot interact through a shared edge, so an odd $\sigma$ acts on each basis element independently.

\subsection{Admissible potentials and the shared-path obstruction}
\label{sec:admissible_obstruction}

The cactus criterion does not depend on the specific choice $\psi(t)=\cosh t-1$; it holds for every admissible edge potential in the following sense.

\begin{definition}
\label{def:admissible}
A function $\psi\colon\R\to[0,\infty)$ is an \emph{admissible edge potential} if
\begin{enumerate}[label=\textup{(\roman*)}]
\item $\psi\in C^2(\R)$, $\psi$ is even, and $\psi(0)=0$;
\item $\psi$ is strictly convex;
\item $\psi''$ is not identically constant.
\end{enumerate}
We write $\sigma:=\psi'$.
\end{definition}

Conditions (i)--(ii) make $\sigma$ odd, strictly increasing, and hence injective, with $\sigma(0)=0$. Condition (iii) excludes the quadratic case $\psi(t)=\tfrac12 c\,t^2$, for which the $\psi$-selector reduces to $\PiH$ on every graph. Besides the principal example $\psi(t)=\cosh t-1$, other admissible potentials include $\psi(t)=|t|^\alpha/\alpha$ for $\alpha>2$ and $\psi(t)=(\cosh t-1)^p$ for $p>1$.

The variational existence theorem of Section~\ref{sec:nonlinear_energy} carries over with almost no change to any admissible $\psi$: the functional is strictly convex, coercive, and produces a unique minimizer whose Euler--Lagrange equation is the $\psi$-coclosed condition $\delta\sigma(z)=0$.

\begin{proposition}
\label{prop:admissible_existence}
Let $\psi$ be admissible with $\sigma=\psi'$. For $h\in C^1(G)$ define
\begin{equation}
\Phi_{\psi,h}(u):=\sum_{e\in E^\to}\psi\bigl(h(e)+(du)(e)\bigr),
\qquad u\in\Hzero.
\label{eq:psi_energy}
\end{equation}
Then $\Phi_{\psi,h}$ has a unique minimizer $u_{\psi,h}\in\Hzero$, and $z_{\psi,h}:=h+du_{\psi,h}$ is the unique representative of $[h]$ satisfying
\begin{equation}
\delta\sigma(z_{\psi,h})=0.
\label{eq:psi_coclosed}
\end{equation}
We write $\Pi_\psi(h):=z_{\psi,h}$.
\end{proposition}

\begin{proof}
\emph{Strict convexity.} For distinct $u,v\in\Hzero$, Lemma~\ref{lem:poincare} gives $du\neq dv$, so strict convexity of $\psi$ on each edge yields $\Phi_{\psi,h}((u+v)/2)<\tfrac12\Phi_{\psi,h}(u)+\tfrac12\Phi_{\psi,h}(v)$.

\emph{Coercivity.} Since $\psi$ is strictly convex, even, and minimized at $0$, one has $\psi(1)>0$; convexity gives $\psi(1)\le t^{-1}\psi(t)$ for $t\ge 1$, and evenness extends this to $\psi(t)\ge\psi(1)(|t|-1)$ for all $t\in\R$. Summing edgewise, using $\|\xi\|_{\ell^1(E^\to)}\ge\|\xi\|_{C^1}$, and applying Lemma~\ref{lem:poincare},
\begin{equation}
\Phi_{\psi,h}(u)\ge \psi(1)\bigl(c_G\|u\|_{C^0}-\|h\|_{C^1}-|E|\bigr)\to\infty
\end{equation}
as $\|u\|_{C^0}\to\infty$ in $\Hzero$. Combined with strict convexity, coercivity produces a unique minimizer $u_{\psi,h}$.

\emph{Euler--Lagrange characterization.} For $\eta\in\Hzero$,
\begin{equation}
D\Phi_{\psi,h}(u)[\eta]
=\sum_{e\in E^\to}\sigma\bigl(h(e)+(du)(e)\bigr)\,(d\eta)(e)
=\langle\delta\sigma(h+du),\eta\rangle_{C^0},
\end{equation}
so, exactly as in Proposition~\ref{prop:stationarity}, critical points are characterized by $\delta\sigma(h+du)=0$. The unique minimizer $u_{\psi,h}$ therefore gives the unique $\psi$-coclosed representative $z_{\psi,h}\in[h]$.
\end{proof}

The analogue of the quadratic gap~\eqref{eq:quadratic_gap} is \emph{not} claimed for a general admissible $\psi$, but Proposition~\ref{prop:admissible_existence} provides all that the cactus criterion needs: a well-defined nonlinear selector $\Pi_\psi$ for every admissible potential.

We now turn to the reverse direction of the criterion. If $G$ is not a cactus, it contains two simple cycles sharing at least one edge, and this overlap produces a linearly coclosed field whose image under $\sigma$ fails to be coclosed. The obstruction is local: it sits at a single vertex where the two overlapping cycles split, and its non-vanishing reduces to a functional equation for the scalar function $\sigma(2t)-2\sigma(t)$ which fails precisely when $\psi$ is non-quadratic.

\begin{lemma}
\label{lem:shared_path_obstruction}
Let $\psi$ be an admissible edge potential with $\sigma=\psi'$, and let $C_1,C_2$ be simple cycles in $G$ sharing a nontrivial path. Then there exist orientations of $C_1,C_2$ whose oriented cycle forms agree on the shared path, and a scalar $t_0>0$, such that $z:=h_{C_1}+h_{C_2}$ satisfies
\begin{equation}
t_0 z\in\ker\delta,
\qquad
\delta\sigma(t_0 z)\neq 0.
\label{eq:obstruction}
\end{equation}
In other words, $t_0 z$ is linearly coclosed but not $\psi$-coclosed.
\end{lemma}

\begin{proof}
Since $\ker\delta$ is a linear subspace, $tz\in\ker\delta$ for every $t\in\R$, so only the second statement requires work.

Let $P$ be the maximal common path of $C_1,C_2$, and let $v$ be an endpoint of $P$ where the two cycles separate. Denote by $e_0$ the edge of $P$ incident to $v$ and by $e_1,e_2$ the distinct non-shared edges of $C_1,C_2$ at $v$. Choose the orientations of $C_1,C_2$ so that both traverse $e_0$ into $v$ and then leave $v$ along $e_1$ and $e_2$ respectively. Writing $\iota(v,e)\in\{+1,-1,0\}$ for the signed incidence indicator ($+1$ if $e$ points into $v$, $-1$ if out, $0$ otherwise), the signed contributions at $v$ are
\begin{equation}
\iota(v,e_0)\,z(e_0)=2,\qquad
\iota(v,e_1)\,z(e_1)=-1,\qquad
\iota(v,e_2)\,z(e_2)=-1,
\end{equation}
and no other support edge of $z$ is incident to $v$. Since $\sigma$ is odd, the divergence of $\sigma(tz)$ at $v$ is therefore
\begin{equation}
(\delta\sigma(tz))(v)=\sigma(2t)-2\sigma(t)=:g(t).
\end{equation}

It remains to show $g\not\equiv 0$ on $[0,\infty)$. If $g\equiv 0$, then $g'(t)=2\psi''(2t)-2\psi''(t)\equiv 0$, giving $\psi''(2t)=\psi''(t)$ for every $t\ge 0$. Iterating this identity yields $\psi''(t)=\psi''(t/2^n)$ for every $n\in\mathbb{N}$; letting $n\to\infty$ and using continuity of $\psi''$ at $0$ forces $\psi''(t)=\psi''(0)$ for every $t\ge 0$, and evenness of $\psi''$ then makes it globally constant, contradicting admissibility~(iii). Hence some $t_0>0$ has $g(t_0)\neq 0$, so $(\delta\sigma(t_0 z))(v)\neq 0$ and in particular $\delta\sigma(t_0 z)\neq 0$.
\end{proof}

\subsection{The cactus criterion}
\label{sec:cactus_theorem}

With the basis of Lemma~\ref{lem:cactus_basis} and the obstruction of Lemma~\ref{lem:shared_path_obstruction} in hand, the main theorem of the paper follows in a few lines.

\begin{theorem}
\label{thm:cactus}
Let $\psi$ be an admissible edge potential. For a finite connected graph $G$, the following are equivalent:
\begin{enumerate}[label=\textup{(\alph*)}]
\item $\Pi_\psi(\omega)=\PiH(\omega)$ for every $\omega\in C^1(G)$;
\item $G$ is a cactus graph.
\end{enumerate}
In particular, the same graph-theoretic condition governs global collapse for every admissible $\psi$ simultaneously.
\end{theorem}

\begin{proof}
\emph{(b) $\Rightarrow$ (a).}
Assume $G$ is a cactus, and let $h_1,\ldots,h_m$ be the basis of $\ker\delta$ from Lemma~\ref{lem:cactus_basis}. Every $z\in\ker\delta$ decomposes uniquely as $z=\sum_j\alpha_j h_j$. Because the supports of the $h_j$ are edge-disjoint, on each edge $e$ belonging to cycle $C_j$ one has $z(e)=\alpha_j h_j(e)$ with $h_j(e)\in\{+1,-1\}$, while on edges outside every cycle $z(e)=0$. Oddness of $\sigma$ therefore gives $\sigma(z(e))=\sigma(\alpha_j)\,h_j(e)$ on each cycle edge and $\sigma(0)=0$ elsewhere, whence
\begin{equation}
\sigma(z)=\sum_{j=1}^m\sigma(\alpha_j)\,h_j,
\qquad
\delta\sigma(z)=\sum_{j=1}^m\sigma(\alpha_j)\,\delta h_j=0
\end{equation}
by Lemma~\ref{lem:cactus_basis}(b). Every element of $\ker\delta$ is thus automatically $\psi$-coclosed. For arbitrary $\omega\in C^1(G)$, the linear Hodge representative $z=\PiH(\omega)$ lies in $\ker\delta\cap[\omega]$, so $\delta\sigma(z)=0$; uniqueness of the $\psi$-coclosed representative in $[\omega]$ (Proposition~\ref{prop:admissible_existence}) then forces $\Pi_\psi(\omega)=z=\PiH(\omega)$.

\emph{(a) $\Rightarrow$ (b).}
Suppose $G$ is not a cactus, so it contains two simple cycles sharing a nontrivial path. Lemma~\ref{lem:shared_path_obstruction} produces $w\in\ker\delta$ with $\delta\sigma(w)\neq 0$. Since $w\in\ker\delta$, one has $\PiH(w)=w$; since $w$ fails to be $\psi$-coclosed, uniqueness gives $\Pi_\psi(w)\neq w$, so $\Pi_\psi\neq\PiH$ on $C^1(G)$.
\end{proof}

Specializing to the reciprocal cost potential recovers the headline statement announced in the introduction.

\begin{corollary}
\label{cor:cactus_cosh}
$\Picc=\PiH$ on all of $C^1(G)$ if and only if $G$ is a cactus graph.
\end{corollary}

\begin{proof}
Apply Theorem~\ref{thm:cactus} with $\psi(t)=\cosh t-1$.
\end{proof}

\begin{remark}
Theorem~\ref{thm:cactus} pinpoints exactly where any nonlinear selector departs from linear Hodge theory: not merely when $\beta_1(G)>0$, but precisely when distinct cycles interact through shared edges. The mechanism is structural in both directions. On a cactus, edge-disjoint cycle supports prevent any nonlinear coupling between cycles and the forward argument works for every odd $\sigma$; on a non-cactus, shared edges feed through $\sigma$ to produce an obstruction, and the non-quadratic condition~(iii) rules out the sole accidental cancellation (the linear $\sigma$) that could make the obstruction vanish. Trees and single cycles are both cactus graphs, so Propositions~\ref{prop:tree} and~\ref{prop:cycle} appear as special cases of Theorem~\ref{thm:cactus}.
\end{remark}

Theorem~\ref{thm:cactus} and Corollary~\ref{cor:cactus_cosh} complete the global picture of the paper: the nonlinear and linear selectors coincide everywhere on $C^1(G)$ precisely when the graph is free of overlapping cycles, and the criterion is independent of the choice of admissible potential. The next section illustrates the quantitative content of the non-cactus case on the simplest possible example.

\section{Worked example: the two-triangle graph}
\label{sec:example}

The cactus criterion is qualitative: it tells us \emph{where} $\Picc$ and $\PiH$ differ but not by how much. This final section gives a quantitative illustration on the smallest graph where a disagreement can occur. On the two-triangle graph---two triangles glued along a common edge---the five-dimensional nonlinear coclosed system reduces, by elementary vertex-local arguments, to a single transcendental equation in one scalar unknown. Solving it numerically, we compare the nonlinear representative to the linear Hodge one and record the self-concordant Newton certificate of Theorem~\ref{thm:newton} on this explicit example.

\subsection{The smallest non-cactus graph}
\label{sec:minimal}

\begin{proposition}
\label{prop:minimal_obstruction}
Among connected simple graphs, the non-cactus examples of smallest size have $4$ vertices and $5$ edges. The two-triangle graph of Figure~\ref{fig:two_triangle} realizes this minimum, and is therefore a smallest connected simple graph on which $\Picc$ can differ from $\PiH$.
\end{proposition}

\begin{proof}
A connected simple graph on at most $3$ vertices is a tree or a triangle, and hence a cactus. A connected simple graph on $4$ vertices with $4$ edges has $\beta_1(G)=|E|-|V|+1=1$, so it is unicyclic and therefore also a cactus. Any non-cactus connected simple graph must therefore satisfy $|V|\ge 4$ and $|E|\ge 5$; the two-triangle graph attains both bounds.
\end{proof}

We now analyze a graph realizing this minimum. Its vertex set is $\{1,2,3,4\}$, and its five oriented edges are
\begin{equation}
a=(1\to 2),\quad b=(2\to 3),\quad c=(1\to 3),\quad d=(3\to 4),\quad e=(4\to 1).
\label{eq:edges}
\end{equation}
The two simple cycles $\{a,b,c\}$ and $\{c,d,e\}$ share the edge $c$, so $G$ is not a cactus and Theorem~\ref{thm:cactus} guarantees the existence of inputs $\omega$ with $\Picc(\omega)\neq\PiH(\omega)$.

\begin{figure}[t]
\centering
\begin{tikzpicture}[
    vertex/.style={circle, draw, fill=black!10, minimum size=20pt, inner sep=0pt, font=\small},
    edge/.style={-{Stealth[length=6pt]}, thick},
    elabel/.style={font=\footnotesize, fill=white, inner sep=1.5pt}
]
\node[vertex] (v1) at (0,0) {$1$};
\node[vertex] (v2) at (2,1.2) {$2$};
\node[vertex] (v3) at (4,0) {$3$};
\node[vertex] (v4) at (2,-1.2) {$4$};

\draw[edge] (v1) -- node[elabel, above left] {$a$} (v2);
\draw[edge] (v2) -- node[elabel, above right] {$b$} (v3);
\draw[edge] (v1) -- node[elabel, above] {$c$} (v3);
\draw[edge] (v3) -- node[elabel, below right] {$d$} (v4);
\draw[edge] (v4) -- node[elabel, below left] {$e$} (v1);
\end{tikzpicture}
\caption{The two-triangle graph. The two simple cycles share the edge $c$, so the graph is not a cactus.}
\label{fig:two_triangle}
\end{figure}
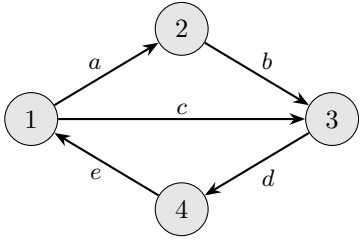

\subsection{Reducing the system to a scalar equation}
\label{sec:scalar_reduction}

Take the input
\begin{equation}
\omega=(5,\,0,\,0,\,0,\,0)
\label{eq:input}
\end{equation}
in the edge order $(a,b,c,d,e)$, and write $z=\omega+d\phi$ with $\phi\in\Hzero$. Only edges $a,b$ are incident to vertex $2$, and only edges $d,e$ to vertex $4$; the nonlinear coclosed equation at those two vertices thus reads $\sinh(z_a)=\sinh(z_b)$ and $\sinh(z_d)=\sinh(z_e)$, which by injectivity of $\sinh$ forces
\begin{equation}
z_a=z_b,
\qquad
z_d=z_e.
\label{eq:vertex24_reduction}
\end{equation}
Combining~\eqref{eq:vertex24_reduction} with the mean-zero constraint $\phi_1+\phi_2+\phi_3+\phi_4=0$ fixes three of the four components of $\phi$ in terms of $\phi_1$,
\begin{equation}
\phi_2=-\tfrac{15}{8},\qquad \phi_3=\tfrac{5}{4}-\phi_1,\qquad \phi_4=\tfrac{5}{8},
\end{equation}
and hence reduces the five edge values to a one-parameter family:
\begin{equation}
z_a=z_b=\tfrac{25}{8}-\phi_1,
\qquad
z_c=\tfrac{5}{4}-2\phi_1,
\qquad
z_d=z_e=\phi_1-\tfrac{5}{8}.
\label{eq:parametrization}
\end{equation}
Every representative of $[\omega]$ satisfying the local coclosed relations at vertices $2$ and $4$ is of this form. The remaining coclosed equation at vertex $1$ then takes the form of a single scalar equation in $\phi_1$; its form differs between the linear and the nonlinear theories.

\subsection{Linear versus nonlinear representative}
\label{sec:lin_nl_comparison}

In the linear theory, the coclosed equation at vertex $1$ is $z_e=z_a+z_c$. Substituting~\eqref{eq:parametrization} gives $\phi_1=5/4$ and therefore
\begin{equation}
\PiH(\omega)=\Bigl(\tfrac{15}{8},\,\tfrac{15}{8},\,-\tfrac{5}{4},\,\tfrac{5}{8},\,\tfrac{5}{8}\Bigr).
\label{eq:piH_ex}
\end{equation}

The nonlinear coclosed equation at vertex $1$ is $\sinh(z_e)=\sinh(z_a)+\sinh(z_c)$, which under~\eqref{eq:parametrization} becomes $F(\phi_1)=0$ for
\begin{equation}
F(t):=\sinh\!\left(t-\tfrac{5}{8}\right)-\sinh\!\left(\tfrac{25}{8}-t\right)-\sinh\!\left(\tfrac{5}{4}-2t\right).
\label{eq:F_def}
\end{equation}
Since $F'(t)=\cosh(t-\tfrac{5}{8})+\cosh(\tfrac{25}{8}-t)+2\cosh(\tfrac{5}{4}-2t)>0$, the function $F$ is strictly increasing, so it has at most one root. The sign changes $F(5/4)<0$ and $F(3/2)>0$ show that a root exists in $(5/4,3/2)$, and numerically
\begin{equation}
\phi_1\approx 1.35690123027,
\end{equation}
which gives
\begin{equation}
\Picc(\omega)\approx(1.76809877,\,1.76809877,\,-1.46380246,\,0.73190123,\,0.73190123).
\label{eq:picc_ex}
\end{equation}

The nonlinear representative differs from the linear one by
\begin{equation}
\|\Picc(\omega)-\PiH(\omega)\|_{C^1}\approx 0.30236234,
\end{equation}
and, consistent with the quadratic gap~\eqref{eq:quadratic_gap} of Theorem~\ref{thm:companion}, achieves a strictly lower $\cosh$-energy: $\sum_e(\cosh\PiH(\omega)_e-1)\approx 5.96610535$ versus $\sum_e(\cosh\Picc(\omega)_e-1)\approx 5.86724188$. The reduced three-parameter form of the two representatives is summarized in Table~\ref{tab:two_triangle}.

\begin{table}[ht]
\centering
\begin{tabular}{@{}lccc@{}}
\toprule
Representative & $z_a=z_b$ & $z_c$ & $z_d=z_e$ \\
\midrule
$\PiH(\omega)$  & $1.87500000$ & $-1.25000000$ & $0.62500000$ \\
$\Picc(\omega)$ & $1.76809877$ & $-1.46380246$ & $0.73190123$ \\
\bottomrule
\end{tabular}
\caption{Linear and nonlinear coclosed representatives of the input $\omega=(5,0,0,0,0)$ on the two-triangle graph.}
\label{tab:two_triangle}
\end{table}

\subsection{Computational certificate}
\label{sec:computational}

Theorem~\ref{thm:newton} supplies an explicit a priori bound on the number of iterations the self-concordant Newton method needs to reach the nonlinear representative. Starting from $\phi^{(0)}=0$, the initial energy is $\Phi_\omega(0)=\cosh 5-1\approx 73.21$, so the damped-phase iteration count is bounded by
\begin{equation}
\left\lceil\frac{\Phi_\omega(0)}{\theta_\star}\right\rceil=\left\lceil\frac{73.21}{\tfrac14-\log\tfrac54}\right\rceil=2726.
\end{equation}
Warm-starting instead from $\phi^{(0)}_{\mathrm{warm}}\in\Hzero$ chosen so that $\omega+d\phi^{(0)}_{\mathrm{warm}}=\PiH(\omega)$ lowers the initial energy to $\Phi_\omega(\phi^{(0)}_{\mathrm{warm}})\approx 5.97$ and the damped-phase bound to
\begin{equation}
\left\lceil\frac{5.97}{\theta_\star}\right\rceil=223.
\end{equation}
These are conservative worst-case bounds, not realistic iteration counts; once the iterate enters the local phase $\lambda_k<1/4$, convergence becomes quadratic and reaches machine precision in a handful of additional steps.

\section{Concluding remarks}

The results of this paper organize naturally along three scales. \emph{Variationally}, the $\cosh$-energy on the mean-zero slice admits a unique minimizer in every cohomology class (Theorem~\ref{thm:companion}), separated from every other representative by an explicit quadratic gap; combined with the self-concordance bound of Theorem~\ref{thm:self_concordance}, this minimizer can be computed by a certified two-phase Newton method with an explicit iteration count (Theorem~\ref{thm:newton}). \emph{Infinitesimally}, the resulting selector $\Picc$ is real analytic, its image is the smooth submanifold $\operatorname{arsinh}(\ker\delta)$, and at every point its differential is a weighted Hodge projector whose edge weights are determined by the output itself (Theorems~\ref{thm:smoothness} and~\ref{thm:weighted_hodge}); at the origin the weights trivialize, $D\Picc(0)=\PiH$, and the first nonlinear correction appears at cubic order (Theorem~\ref{thm:cubic}, Corollary~\ref{cor:multivar_cubic}). \emph{Globally}, the question of when $\Picc$ collapses to $\PiH$ on all of $C^1(G)$ is settled by a sharp combinatorial criterion (Theorem~\ref{thm:cactus}): agreement holds for every admissible edge potential simultaneously if and only if $G$ is a cactus graph, and the smallest connected simple counterexample---the two-triangle graph of Section~\ref{sec:example}---already exhibits a quantifiable departure.

What the cactus criterion really identifies is \emph{where nonlinearity first has somewhere to act}. On a cactus, the simple cycles occupy pairwise edge-disjoint supports, so an odd nonlinearity $\sigma$ acts on each cycle as a one-dimensional scalar rescaling and cannot couple distinct cycles; the linear condition $\delta z=0$ on each cycle already implies the nonlinear condition $\delta\sigma(z)=0$. On a non-cactus graph, any two overlapping cycles share at least one edge, and the elementary obstruction $\sigma(2t)\neq 2\sigma(t)$---ruled out only in the linear case---translates shared edges into a nonzero $\psi$-divergence. The criterion therefore depends only on cycle overlap in $G$, not on any further detail of the potential beyond admissibility.

Read alongside the cost-first discrete ledger of~\cite{PSTW2026Ledger}, the criterion supplies the complementary combinatorial input. There the canonical cost $J(x)=\tfrac12(x+x^{-1})-1$ governs a directed-graph ledger of recognition events, and a time-aggregated cycle-closure hypothesis is invoked to reconstruct a linear scalar potential on each connected component (their Theorems~T3 and T4). Theorem~\ref{thm:cactus} sharpens this picture from the opposite side: it identifies the precise combinatorial graphs on which the nonlinear coclosed selector built from the same cost coincides with the linear Hodge projector $\PiH$ \emph{without} appealing to cycle closure. Cactus graphs are exactly the topologies on which the linear and nonlinear ledger pictures agree, for every admissible potential.

Several directions remain natural. First, dropping evenness of $\psi$ breaks the odd symmetry of $\Picc$, invalidating the forward cactus argument through $\pm 1$-valued cycle forms; an analogue for strictly convex non-even potentials would likely require a different combinatorial model of $\ker\delta$. Second, moving to weighted or directed graphs makes the coclosed equation asymmetric and may alter the geometry of the nonlinear slice. Third, moving further to higher-dimensional simplicial complexes raises the prospect of a higher-order cactus criterion in which the obstruction involves overlapping higher-dimensional cycles rather than shared graph edges. Finally, on the quantitative side, the cubic expansion of Corollary~\ref{cor:multivar_cubic} invites sharper global estimates on $\Picc(\omega)-\PiH(\omega)$ away from the origin, and the a priori Newton counts in Section~\ref{sec:computational} are likely far from tight in practice.


\authorcontributions{Conceptualization, S.P.-G., A.T., and J.W.; methodology, S.P.-G., A.T., and J.W.; formal analysis, S.P.-G., A.T., and J.W.; writing---original draft preparation, S.P.-G., J.W.; writing---review and editing, S.P.-G., A.T., and J.W. All authors have read and agreed to the published version of the manuscript.}

\bibliographystyle{apsrev4-2}
\bibliography{Bib}

\end{document}